\documentclass[11pt]{amsart}

\usepackage{alltt}

\usepackage{marginnote}

\usepackage{amsfonts}
\usepackage{euscript}

\usepackage{latexsym,epsfig,verbatim}
\usepackage{amsmath,amsthm,amssymb}
\usepackage{colonequals}
\usepackage{units}

\usepackage{url}
\usepackage{color}
\usepackage[dvipsnames]{xcolor}
\definecolor{pastelpink}{RGB}{233,175,221}
\definecolor{barneypurple}{RGB}{153,85,255}
\definecolor{oceanblue}{RGB}{0,0,255}

\usepackage[colorlinks,citecolor=blue,linkcolor=red!80!black]{hyperref}
\usepackage[nameinlink]{cleveref}

\usepackage{comment}
\usepackage{amscd}
\usepackage{hyperref}
\usepackage{accents}

\usepackage{tikz}
\usetikzlibrary{matrix,arrows,decorations.pathmorphing}
\theoremstyle{plain}
\newtheorem{theorem}{Theorem}[section]
\newtheorem{lemma}[theorem]{Lemma}
\newtheorem{corollary}[theorem]{Corollary}

\crefname{claim}{Claim}{Claims}
\newtheorem*{claim*}{Claim}

\theoremstyle{definition}
\newtheorem{defn}[theorem]{Definition}

\newtheorem{question}[theorem]{Question}

\crefname{convention}{Convention}{Conventions}

\newcommand{\calF}{\mathcal{F}}

\newcommand{\stl}{\mathrm{stl}}
\newcommand{\tl}{\mathrm{tl}}
\usepackage{enumitem}
 
\DeclareMathOperator{\Map}{Map}

\title{Stable specific torsion length and periodic mapping classes}
\author{Elizabeth Field and Yvon Verberne}

\address{Elizabeth Field \\ Department of Mathematics \\ University of Utah \\ 115 South 1400 East \\ Salt Lake City, UT 84112}
\email{field@math.utah.edu}

\address{Yvon Verberne \\ School of Mathematics\\ Georgia Institute of Technology \\ 686 Cherry St. \\ Atlanta, GA 30332}
\email{yverberne3@gatech.edu}

\date{\today}

\begin{document}

\begin{abstract}
We show that for any periodic mapping class, there is some power which maps a nonseparating, simple closed curve to a distinct, disjoint nonseparating curve. As an application of this result, we introduce the notion of stable specific torsion length of a group element and show that the stable specific torsion length of a Dehn twist is bounded above by six.
\end{abstract}

\maketitle

\section{Introduction}

Let $S_g$ denote a connected, closed, orientable surface of genus $g$.
The mapping class group of $S_g$, $\Map(S_g)$, is
the group of orientation-preserving homeomorphisms of the surface up to isotopy. 
A fundamental theorem in the study of mapping class groups states that for any two disjoint, nonseparating, simple closed curves $a$ and $b$ on a surface, there exists a homeomorphism $\phi$ of the surface such that $\phi(a) = b$. However, it was previously unknown whether a fixed mapping class has a representative which can eventually map a nonseparating curve to one which is disjoint from itself. We resolve this question for periodic mapping classes.

\newcommand{\tDisjointCurves}{
Let $S_g$ be a surface of genus $g \geq 3$, and let $\phi$ be a periodic mapping class of order $d\geq 2$. 
Then, there exists some power $k \geq 1$ and a nonseparating, simple closed curve $c$ such that $i(c, \phi^{k}(c)) = 0$ and $c \neq \phi^{k}(c)$.
}

\begin{theorem}\label{T:DisjointCurves}
\tDisjointCurves
\end{theorem}

To prove our main result, we utilize theorems of Klein \cite{Klein} and Kulkarni \cite{Kulk} which characterize periodic maps of a surface in terms of the quotient orbifold. 
Based on our result, we ask the following two questions.
The first asks whether it is necessary to take a power of the given periodic map in order to take a nonseparating curve to a disjoint curve. The second asks whether the result of \Cref{T:DisjointCurves} holds for mapping classes which are not periodic.

\begin{question}
Fix a periodic mapping class $\phi\in \Map(S_g)$ of order $d\geq 2$. Do there exist distinct, disjoint, nonseparating, simple closed curves $a$ and $b$ such that $\phi(a) = b$?
\end{question}

\begin{question} 
For which mapping classes $f \in \Map(S_g)$ do there exist distinct, disjoint, nonseparating, simple closed curves $a$ and $b$ such that $f^k(a) = b$ for some $k\geq 1$? 
\end{question}

\noindent \textbf{Stable specific torsion length.}
Let $f \in \Map(S_g)$, and let $G$ denote a generating set of $\Map(S_g)$.
The \textit{word length} of $f$, denoted by $|f|_{G}$, measures the smallest number of elements in the generating set $G$ needed to express $f$.
Since the mapping class group is non-elementary, there is no uniform upper bound on the word length of an element $f \in \Map(S_g)$.
Therefore, it is interesting to study the growth of the word length as we pass to higher powers of $f$. 
The \textit{stable word length} of an element $f \in \Map(S_g)$ with respect to the generating set $G$ is defined to be
\[
||f||_G := \lim_{n \rightarrow \infty} \frac{|f^{n}|_G}{n}.
\]
While more is known about stable commutator length \cite{Cal2, CF, Chen}, stable word length has been studied by Calegari in \cite{Cal} and more recently by Ye in \cite{Ye21}.

Much work has been done in studying various generating sets for $\Map(S_g)$. The first generating sets found for $\Map(S_g)$ consisted of Dehn twists. In \cite{Dehn1}, Dehn first proved that $2g(g-1)$ Dehn twists generate $\Map(S_g)$. Lickorish \cite{Lick} later showed that $\Map(S_g)$ is generated by Dehn twists about $3g - 1$  nonseparating simple closed curves. Humphries \cite{Humph} reduced this generating set to $2g + 1$ Dehn twists, and further proved that this number is minimal for a generating set consisting of Dehn twists. More recent work has involved finding generating sets for $\Map(S_g)$ involving torsion elements. For instance, McCarthy-Papadopoulos showed that $\Map(S_2)$ is normally generated by two involutions, while $\Map(S_g)$ is normally generated by a single involution when $g\geq 3$ \cite{MP}. Korkmaz \cite{Kork} showed that when $g\geq 3$, $\Map(S_g)$ is generated by two elements of order $4g + 2$. For more results about specific generating sets of $\Map(S_g)$ see, for example, \cite{Kass, Luo, Mac, Mond, Waj}.

Let $T$ denote the subset of the mapping class group consisting of all torsion elements. We note that by the above results, $T$ is a generating set for $\Map(S_g)$ when $g\geq 2$. 
The \textit{torsion length} of an element $f\in\Map(S_g)$, denoted by $\mathrm{tl}(f)$, is the word length of $f$ with respect to the set of all torsion elements in $\Map(S)$. The notion of torsion length for elements of $\Map(S)$ was introduced by Brendle and Farb in \cite{BF} where they ask whether there exists a constant $C$ such that every element of $\Map(S)$ can be written as a product of at most $C$ torsion elements. This question was answered in the negative by Korkmaz in \cite{Kork} as well as by Kotschick in \cite{Kots}. 

The \textit{stable torsion length} of $f$, introduced by Korkmaz in \cite{Kork} and Kotschick in \cite{Kots} and further studied by Avery and Chen in \cite{AC}, is defined to be the stable word length of $f$ with respect to all of the torsion elements in $\Map(S)$; that is
\[
\mathrm{stl}(f) := \lim_{n \rightarrow \infty} \frac{\mathrm{tl}(f^n)}{n}.
\] 
Since torsion length is sub-additive, $\stl(f)\leq \tl(f)$ for all $f \in \Map(S_g)$.

In work of Brendle and Farb, it is shown that the Dehn twist about any nonseparating, simple closed curve can be written as the product of two torsion elements \cite{BF}.
This shows that the stable torsion length of a Dehn twist is bounded above by two.
The proof of this result is constructive, and so the torsion elements required to write the Dehn twist depend upon the curve about which you are twisting. In \cite{LM}, Lanier-Margalit show that any periodic element which is not the hyperelliptic involution normally generates $\Map(S_g)$ for all $g\geq 3$. 
Therefore, one can ask the following question.

\begin{question}\label{Question:StableTorsionDehnTwist}
Given a fixed torsion element $t\in \Map(S_g)$, how many conjugates of $t$ are required to write the Dehn twist about any nonseparating, simple closed curve?
\end{question}

This question motivates the following definition, suggested by Dan Margalit.

\begin{defn} 
Let $t\in \Map(S_g)$ be a torsion element, and let $C_t$ denote the set of all conjugates of $t$ in $\Map(S_g)$. Given any element $f\in \Map(S_g)$, we define the \textit{specific torsion length of $f$ with respect to $t$} to be 
$$\mathrm{tl}_t(f) := \inf \{n \mid f = t_1\cdots t_n \text{, where } t_i\in C_t \}, $$ and the \textit{stable specific torsion length of $f$ with respect to $t$} to be 
$$\mathrm{stl}_t(f) := \lim_{n\to \infty} \frac{\mathrm{tl}_t(f^n)}{n}. $$ We make the convention that $\tl_t(f) = \infty$ if $f$ cannot be written as a product of conjugates of $t$.
\end{defn}

Using \Cref{T:DisjointCurves}, we are able to answer \Cref{Question:StableTorsionDehnTwist}.

\newcommand{\tSixConjugates}{
Let $g \geq 3$ and let $\phi \in \Map(S_g)$ be a non-trivial periodic mapping class which is not the hyperelliptic involution.
Then, the Dehn twist about any fixed nonseparating, simple closed curve in $S_g$ can be written as a product of $6$ distinct conjugates of $\phi$.
}

\begin{theorem}\label{T:SixConjugates}
\tSixConjugates
\end{theorem}

As a corollary, we show that the stable specific torsion length for a Dehn twist about any simple closed curve $c$ with respect to most torsion elements is bounded above by $6$.

\begin{corollary}
Let $g \geq 3$, let $T_c$ be the Dehn twist about any fixed nonseparating curve $c$, and let $t \in \Map(S_g)$ be any torsion element which is not the hyperelliptic involution.
Then, $\stl_t(T_c) \leq 6$.
\end{corollary}

Although we are able to find an upper bound for the stable specific torsion length of a Dehn twist with respect to most periodic mapping classes, there are still many questions which remain.

\begin{question}\label{Question:SSTl_Optimal}
Given a torsion element $t\in \Map(S_g)$ and a Dehn twist $T_c$ about any nonseparating curve $c$, what is $\stl_t(T_c)$? 
\end{question}

\begin{question}\label{Question:SSTl_DoesItVary}
Does the stable specific torsion length of a Dehn twist depend on the specific torsion element chosen?
\end{question}

\begin{question}\label{Question:SSTl_OtherElements}
Is there a uniform upper bound on stable specific torsion length for other elements of $\Map(S_g)$, or is stable specific torsion length unbounded on $\Map(S_g)$? Does this answer depend on the Nielsen-Thurston type of the mapping class?
\end{question}

Recently, Lanier proved that given any periodic normal generator $\phi$ of $\Map(S_g)$ of order at least 3,  $\Map(S_g)$ is generated by $60$ conjugates of $\phi$ \cite[Theorem 1.4]{Lanier}. This answered a question asked by Lanier and Margalit \cite[Question 3.4]{LM}. 
In the proof of this result, Lanier shows that a Dehn twist about a nonseparating curve can be generated by $12$ conjugates of $\phi$. 
As \Cref{T:SixConjugates} shows that a Dehn twist about a nonseparating curve can be generated by only $6$ conjugates of $\phi$, we can improve Lanier's bound and show that it suffices to use 54 conjugates of $\phi$ to generate $\Map(S_g)$. \\

\noindent \textbf{Acknowledgements:} Both authors would like to thank Dan Margalit for suggesting the application to stable torsion length. The authors would also like to thank Chris Leininger, Mahan Mj, and Kasra Rafi for many helpful conversations. The first author was partially supported by NSF grants DMS-1840190 and DMS-2103275. 
The second author was supported by the National Science Foundation under Grant No. DMS-1928930 while participating in
a program hosted by the Mathematical Sciences Research Institute in Berkeley, California, during the Fall 2020 semester. The second author was also partially supported by an NSERC-PDF Fellowship.

\section{Proof of main result}\label{Sec:ProofOfMainTheorem}

In this section, we prove \Cref{T:DisjointCurves} which shows that a representative of a periodic mapping class will always map a nonseparating simple closed curve to a distinct, disjoint, nonseparating curve. 
In order to prove this theorem, we use results from Klein \cite{Klein} and Kulkarni \cite{Kulk} which utilize orbifolds to characterize periodic maps of order two, and order greater than two, respectively.

\medskip

\smallskip

\noindent{\bf  \Cref{T:DisjointCurves}}
    {\em \tDisjointCurves } 

\begin{proof}
Let $S_g$ be a surface of genus $g \geq 3$.
Let $\phi$ be a periodic mapping class with order at least 2. 
We fix a standard representative of the mapping class $\phi$ and also denote this representative homeomorphism by $\phi$. 

We consider the following three cases which describe the ways in which $\langle \phi \rangle$ can act on $S_g$, as was done in \cite[Proposition 3.1]{LM}:
\begin{enumerate}
    \item The action of $\langle\phi\rangle$ is free;
    \item The action of $\langle \phi \rangle$ is not free and $\phi$ has order 2; and
    \item The action of $\langle \phi \rangle$ is not free and $\phi$ has order greater than 2.
\end{enumerate}

\noindent \textit{Case (1):} Suppose the action of $\langle\phi\rangle$ is free.
Then, this action must be a covering action.
Since $g\geq 3$, there exists a fundamental domain for the action which contains genus.
Fix such a fundamental domain, $\calF$. Since $\calF$ contains genus, we can find a simple, closed, nonseparating curve $c$ in $\calF$ which is not homotopic to a boundary component of $\calF$. Then, $i(c, \phi(c)) = 0$ and $c \neq \phi(c)$.

\noindent \textit{Case (2):} Suppose the action of $\langle \phi \rangle$ is not free and $\phi$ has order 2. 
In \cite{Klein}, Klein gives a classification of such homeomorphisms.
In a more modern approach to the this work, Dugger explicitly constructs the $4+\lceil \frac{g}{2} \rceil$ involutions (up to conjugation) of a surface of genus $g \geq 0$; see \cite[Theorem 5.7]{Dugger}.
From this result, we see that in all cases but the hyperelliptic involution, orientation preserving involutions have a fundamental domain which contains genus. Since conjugation is a homeomorphism, it preserves the topological type of the fundamental domain. Thus, we can find a curve whose image is taken to a disjoint curve as in Case (1). 
In the case of the hyperelliptic involution, we choose the curve shown in \Cref{Fig:Case2Hyperelliptic}.

\begin{figure}[ht]
\includegraphics[width=300\unitlength]{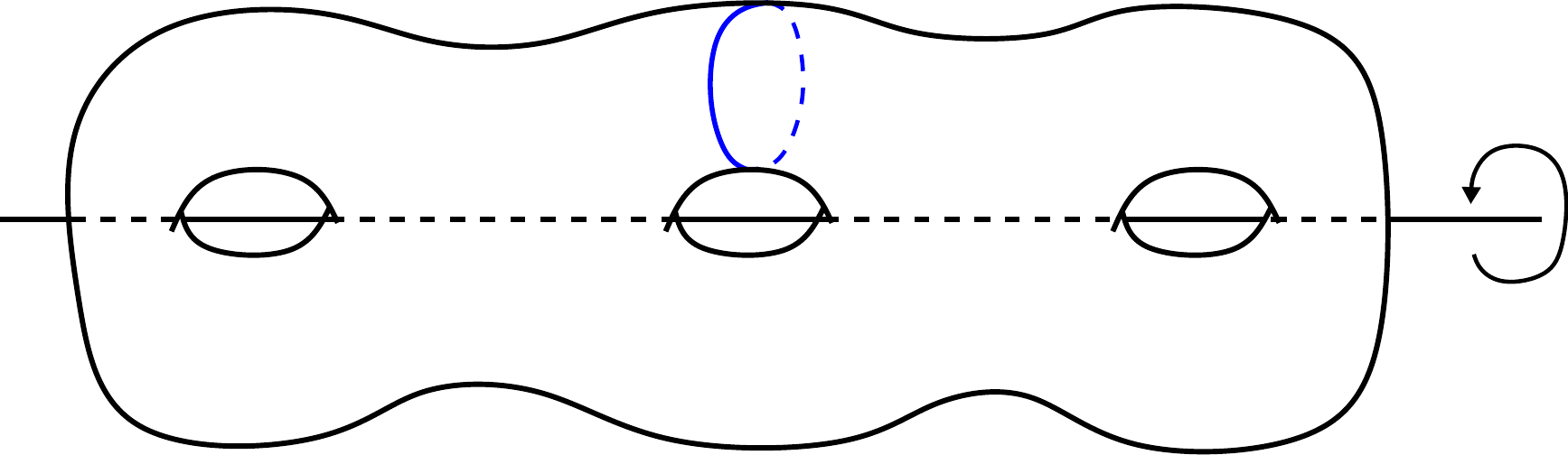}
\put(10,42){$\phi$}
\caption{Curve which maps to a disjoint curve by the hyperelliptic involution.}
\label{Fig:Case2Hyperelliptic}
\end{figure}

\noindent \textit{Case (3):}
Finally, suppose the order of $\phi$ is greater than $2$ and the action of $\langle \phi \rangle$ on $S_g$ is not
free. Thus, some power of $\phi$ has a fixed point. 
If $\phi$ is a root of the hyperelliptic involution, then by a result of Lanier-Margalit (\cite[Lemma 3.2]{LM}), there exists a power of $\phi$ that is neither the identity nor the hyperelliptic involution which has a standard $\Map(S_g)$-representative with a fixed point.
Thus, we can pass to this power of $\phi$ which has a fixed point and which is neither the identity nor the hyperelliptic involution.

In \cite[Theorem 2]{Kulk}, Kulkarni states that if $\phi$ is a finite-order homeomorphism of $S_g$ that has
a fixed point, then $S_g$ can be represented as a quotient space of some regular $n$-gon in such a way
that $\phi$ is realized as a rotation of the $n$-gon by some multiple of $2\pi/n$. 
We will 
use this representation to construct a nonseparating, simple closed curve $c$ for which $i(\phi^k(c), c) = 0$ and $\phi^k(c) \neq c$ for some $k\geq 1$. 

There are two cases we must consider. 
The first case is when the polygon has $4g+2$ vertices and opposite edges of the polygon are identified. 
The other case is when the polygon is some $2n$-gon and at least one edge is identified to another edge within the same half of the polygon.

We first consider the polygon with $4g+2$ vertices with opposite edge gluings.
Label the vertices of this polygon in a counterclockwise ordering with the labels $0, 1, \ldots, 4g+1$, as in \Cref{Fig:4g+2OtherMaps}. We claim that the curve $c$ shown in \Cref{Fig:4g+2OtherMaps} is the required curve in this case. This is the curve which starts on the left half of the edge between vertices 0 and 1, travels to the right half of the edge between vertices 2 and 3, comes out on the edge between the vertices labelled by $2g + 3$ and $2g + 4$, and then closes up by looping around the edge between the vertices labelled $2g + 2$ and $2g + 3$. 
If $\phi$ has degree $d$, then $\phi$ will map the vertex labeled $0$ to the set of vertices $\left\{ 0, \frac{4g + 2}{d}, \frac{2(4g + 2)}{d}, \ldots \frac{(d-1)(4g + 2)}{d} \right\} $, in some order. 
Hence, there is some power of $\phi$ which will map the vertex labeled $0$ to the vertex labeled $\frac{4g+2}{d}$. 
Thus, a fundamental domain for $\phi$ is the polygon with vertices at $0$, the center of the $4g + 2$-gon, and $\frac{4g + 2}{d}$. 
Note that if 
$$\frac{4g + 2}{d} + 3 \leq \frac{4g + 2}{2}, $$ 
then some power of $\phi$ will map the curve $c$ in \Cref{Fig:4g+2OtherMaps} off of itself.
Since we are assuming that $g\geq 3$, this inequality always holds as long as $\phi$ has order at least $4$. If $\phi$ has order $3$, then this inequality is satisfied as long as $g\geq 4$. If $g=3$, then there are no maps of order $3$ for this particular gluing, because $3$ does not divide $4g + 2$ when $g = 3$. Thus, the theorem holds in the case that the polygon has $4g + 2$ edges with opposite edges identified.

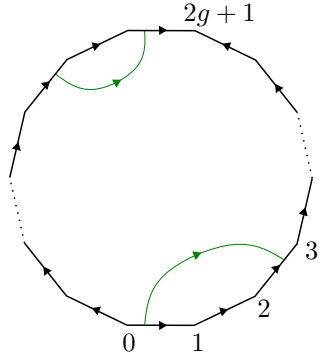
\begin{figure}[ht]
\includegraphics[width=115\unitlength]{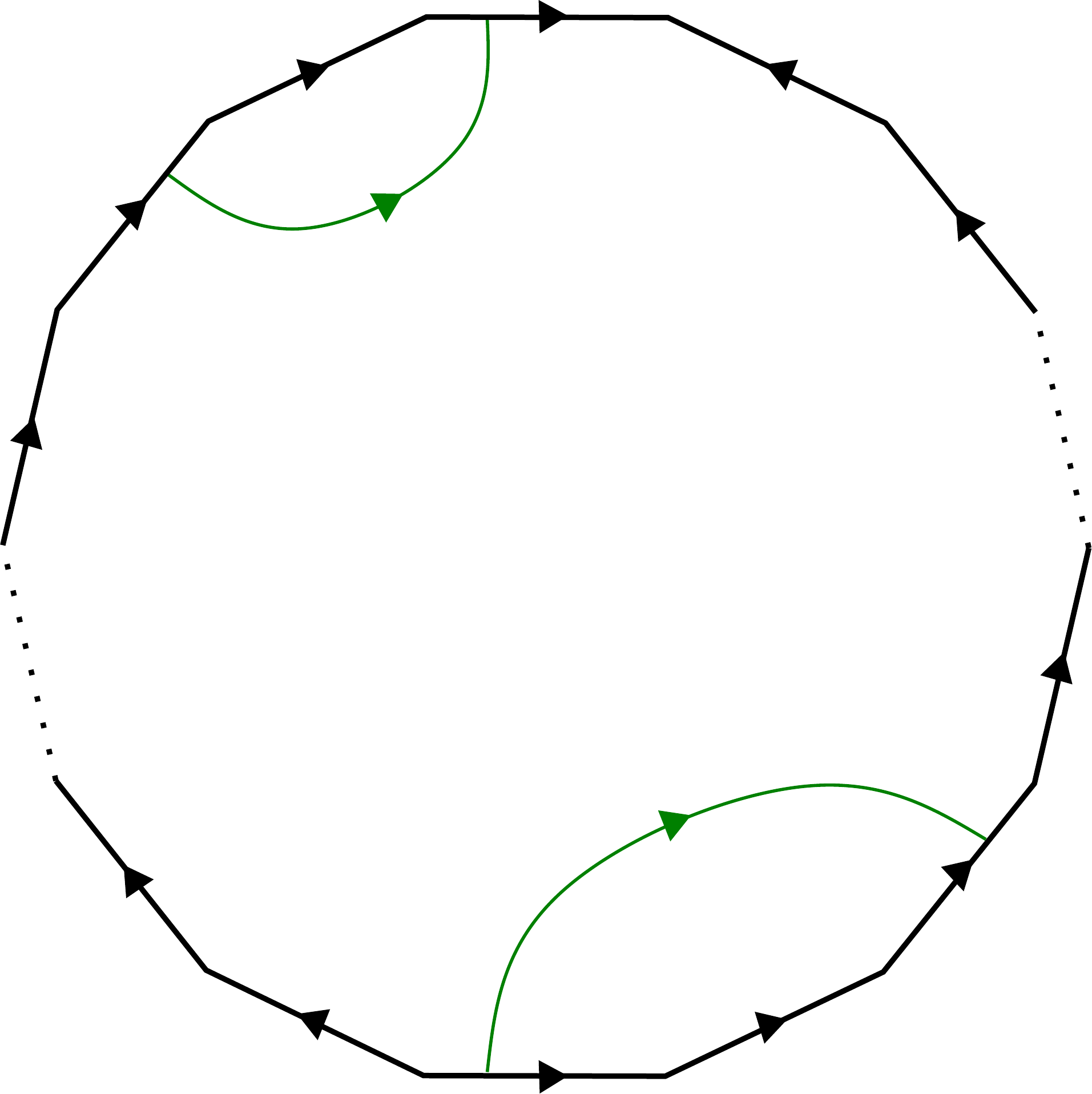}
\put(-72,-8){\small{0}}
\put(-46,-8){\small{1}}
\put(-21,5){\small{2}}
\put(-3,27){\small{3}}
\put(-49,117){\small{$2g+1$}}
\caption{Identification diagram for $S_g$ where opposite sides of a $4g + 2$-gon are identified. The curve $c$ is drawn for a periodic map of order $d \geq 3$.}
\label{Fig:4g+2OtherMaps}
\end{figure}

Finally, suppose $S_g$ is represented as a quotient space of a regular $2n$-gon where at least one edge glues to an edge in the same half. 
Let $e_1$ and $e_2$ denote two edges which glue together on the same side of the polygon. 
We say that $e_1$ and $e_2$ are \textit{minimal} if no two edges between $e_1$ and $e_2$ on that half of the polygon glue to each other.
Choose such a minimal pair of edges and let $c$ denote the simple closed curve which has its endpoints on these two edges. 
First, note that if $e_1$ and $e_2$ are adjacent, then $S_g$ would have had a cone point, which it does not. 
Thus, there must be some edge between $e_1$ and $e_2$ which glues to the other half of the polygon. 
Therefore, $c$ is both essential and non-separating. 
As $\phi$ has order $d\geq 3$, there exists a power $k\geq 1$ of $\phi$ which maps $c$ to a distinct, disjoint curve $\phi^k(c)$ on the other half of the polygon.

\end{proof}

A \textit{bounding pair} is a pair of curves $\alpha, \beta\in S_g$ which are disjoint, nonseparating and homologous. Careful observation of the proof of \Cref{T:DisjointCurves} shows that whenever $\phi$ is not the hyperelliptic involution, the curves $c$ and $\phi^k(c)$ can be chosen so that they do not form a bounding pair. 
We codify this in the following corollary as this observation is key to the proof of \Cref{T:SixConjugates}.

\begin{corollary}\label{Cor:NoBoundingPair}
Let $\phi\in \Map(S_g)$ be a periodic mapping class of order $d\geq 2$. If $\phi$ is not the hyperelliptic involution, then the curves $c$ and $\phi^{k}(c)$ in \Cref{T:DisjointCurves} can be chosen so that they do not form a bounding pair.
\end{corollary}

\begin{proof}
In cases (1) and (2) of the proof of \Cref{T:DisjointCurves}, we may choose the curve $c$ to go around some genus which is mapped to a different genus by $\phi$. Thus, the curves $c$ and $\phi^k(c)$ may be chosen so that they do not form a bounding pair. 
Now suppose we are in case (3) of \Cref{T:DisjointCurves}. 
In the first setting, where the polygon has $4g + 2$ edges with opposite edges identified, the curve $c$ is explicitly chosen so that it is nonseparating. We note that the curve $c$ and its image $\phi^k(c)$ separate the (unglued) polygon into 5 sections: 1 ``central'' section, two ``peripheral'' sections which are cut off by $c$, and two ``peripheral'' sections which are cut off by $\phi^k(c)$. The curve $c$ was explicitly constructed so that the two ``peripheral'' sections are not separated from the central section of the polygon; see \Cref{Fig:4g+2OtherMaps}. As the curve $\phi^k(c)$ cuts away two additional peripheral sections from the polygon which are still not separated from the central section, the curves $c$ and $\phi^k(c)$ cannot form a bounding pair.

Now suppose we are in the setting where the polygon is some regular $2n$-gon and two edges glue together in the same half. 
Suppose that the curve $c$ and $\phi^k(c)$ form a bounding pair. 
Then, since $i(c, \phi^k(c))\neq 0$, this implies that there must be some genus in the subsurface cut off by $c$ and $\phi^k(c)$. 
As $\phi$ acts on the $2n$-gon by rotation, this implies that the genus cut off by $c$ and $\phi^k(c)$ must get taken off itself by some power of $\phi$. 
So, as before, we can choose a curve $d$ which wraps around the genus and thus gets taken disjoint from itself under some power of $\phi$. 
\end{proof}

\section{Stable specific torsion length of Dehn twists}\label{Sec:Application}

In this section, we prove \Cref{T:SixConjugates} and thereby show that the stable specific torsion length of a Dehn twist is bounded above by 6.

\begin{figure}[ht]
\includegraphics[width=300\unitlength]{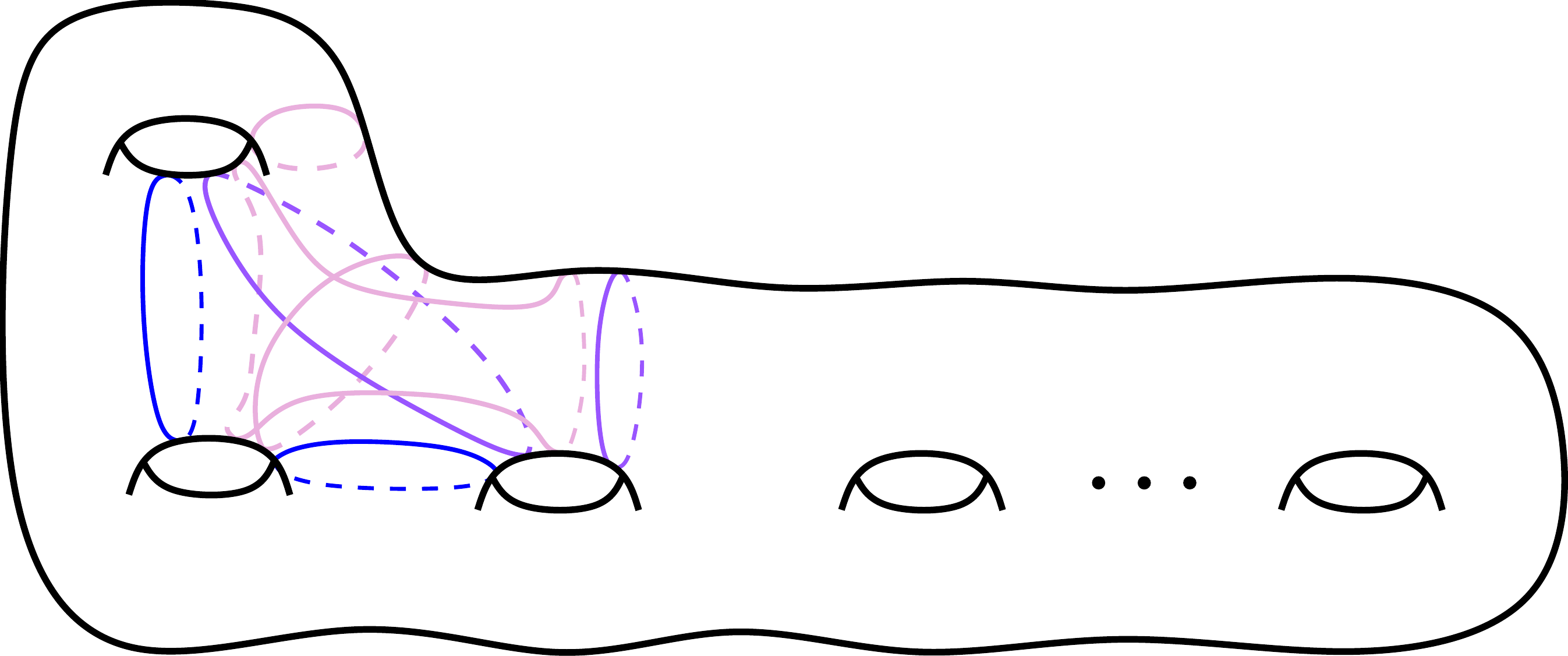}
\put(-175,50){\textcolor{barneypurple}{$\gamma_2$}}
\put(-249,109){\textcolor{pastelpink}{$x_1$}}
\put(-287,63){\textcolor{oceanblue}{$\alpha_2$}}
\put(-232,25){\textcolor{oceanblue}{$\alpha_1$}}
\put(-221, 80){\textcolor{pastelpink}{$x_3$}}
\put(-198, 78){\textcolor{pastelpink}{$x_2$}}
\caption{An embedded lantern in a surface $S_g$.}
\label{Fig:EmbeddedLantern}
\end{figure}
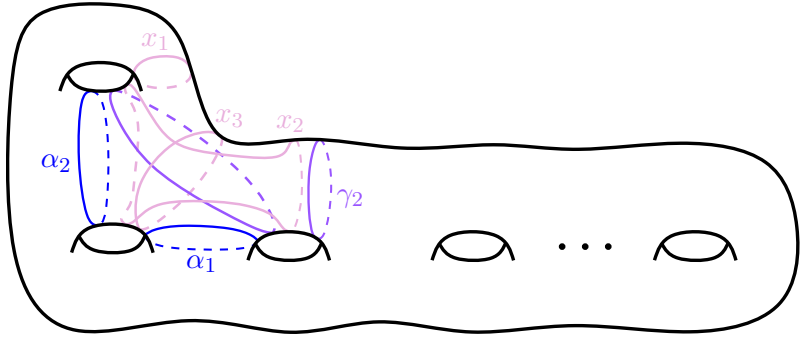

We call a collection of curves $\{\alpha_1, \alpha_2, \gamma_1, \gamma_2, x_1, x_2, x_3\}$ on the surface a \textit{lantern} (see \Cref{Fig:EmbeddedLantern}) if they satisfy the following lantern relation (\cite[Proposition 5.1]{FM}):
\begin{equation}\label{Eqn:Lantern}
T_{\alpha_1} = T_{\gamma_1}T_{\gamma_2}^{-1}T_{x_3}T_{x_1}^{-1}T_{x_2}T_{\alpha_2}^{-1}.
\end{equation}
We will use the aforementioned lantern relation to prove the following lemma which allows us to construct a Dehn twist about the nonseparating curve $\alpha_1$.

\begin{lemma}\label{Lem:LanternToTwist}
Suppose we are given a subsurface $L$ in $S_g$ containing an embedded lantern, as in \Cref{Fig:EmbeddedLantern}, and elements $f$, $g$, and $h$ in $\mathrm{Map}(S_g)$ such that
\begin{enumerate}
    \item $f(\gamma_1) = \gamma_2$;
    \item $g(x_3) = x_1$; and
    \item $h(x_2) = \alpha_2$.
\end{enumerate}
Then, the Dehn twist $T_{\alpha_1}$ may be written as a product in $f^{-1}$, $g^{-1}$, $h^{-1}$, an element conjugate to $f$, an element conjugate to $g$, and an element conjugate to $h$.
\end{lemma}

\begin{proof}
Let $L$ be a subsurface in $S_g$ containing the embedded lantern shown in \Cref{Fig:EmbeddedLantern}, and let $f, g, h\in \Map(S_g)$ satisfy the above conditions. We will use these conditions on $f$, $g$, and $h$ to rewrite the lantern relation, Equation \ref{Eqn:Lantern}, as follows:

\begin{equation*}
    \begin{aligned}
    T_{\alpha_1} &= T_{\gamma_1}T_{\gamma_2}^{-1}T_{x_3}T_{x_1}^{-1}T_{x_2}T_{\alpha_2}^{-1} \\
    &= T_{f^{-1}(\gamma_2)}T_{\gamma_2}^{-1}T_{g^{-1}(x_1)}T_{x_1}^{-1}T_{h^{-1}(\alpha_2)}T_{\alpha_2}^{-1}\\
    &= (f^{-1}T_{\gamma_2}f) T_{\gamma_2}^{-1} (g^{-1}T_{x_1}g)T_{x_1}^{-1} (h^{-1} T_{\alpha_2} h) T_{\alpha_2}^{-1}\\
    &= f^{-1} (T_{\gamma_2}fT_{\gamma_2}^{-1}) g^{-1}(T_{x_1}gT_{x_1}^{-1}) h^{-1} (T_{\alpha_2} h T_{\alpha_2}^{-1})
    \end{aligned}
\end{equation*}
This is a product of $f^{-1}, g^{-1}, h^{-1}$, and  conjugates of $f$, $g$, and $h$, as desired.
\end{proof}
While the above lemma is stated for a Dehn twist about the curve $\alpha_1$, an application of the following change of coordinates principle (see \cite[Section 1.3.3]{FM}) will allow us to write the Dehn twist about any nonseparating, simple closed curve as a conjugate of this product. 

\begin{lemma}[Change of coordinates principle]\label{Lem:ChangeOfCoordiantes}
If $\{\alpha_1, \alpha_2\}$ and $\{\beta_1, \beta_2\}$ are two pairs of disjoint, nonseparating, simple closed curves such that the cut surfaces $S \setminus \{\alpha_1, \alpha_2 \}$ and $S \setminus \{ \beta_1, \beta_2 \}$ are homeomorphic, then there is a homeomorphism $\psi : S \rightarrow S$ which maps the pair $\{\alpha_1, \alpha_2\}$ to the pair $\{\beta_1, \beta_2\}$. 
\end{lemma}

We note that whenever $\{\alpha_1, \alpha_2\}$ and $\{\beta_1, \beta_2\}$ are not bounding pairs, the cut surfaces $S \setminus \{\alpha_1, \alpha_2\}$ and $S\setminus \{\beta_1, \beta_2\}$ are necessarily homeomorphic. 
We now use \Cref{Lem:LanternToTwist} to prove that $6$ conjugates of $\phi$ suffice to build a Dehn twist about any nonseparating curve.

\medskip

\smallskip

\noindent{\bf \Cref{T:SixConjugates}}
    {\em \tSixConjugates }

\begin{proof}
Since $g \geq 3$, we can fix an embedding of seven nonseparating simple closed curves in  the surface $S_g$ such that they satisfy the lantern relation 
\[
T_{\alpha_1} T_{\alpha_2} T_{x_1} T_{\gamma_2} = T_{\gamma_1} T_{x_3} T_{x_2}.
\]
If we can show that there exist elements which are each  conjugate to some power of $\phi$ such that the conditions in \Cref{Lem:LanternToTwist} hold, then we have proven our claim.

We first show there exists an element $\psi_{f}\in \Map(S_g)$ and a power $i \in \mathbb{Z}$ such that $\psi_{f}^{-1} \phi^{i} \psi_{f}(\gamma_1) = \gamma_2$, or equivalently, that $\phi^{i} \psi_{f}(\gamma_1) = \psi_{f}(\gamma_2)$.
We note that the curves $\gamma_1$ and $\gamma_2$ are disjoint and do not form a bounding pair.
\Cref{T:DisjointCurves} together with \Cref{Cor:NoBoundingPair} show that we can find disjoint curves and $a$ and $b = \phi^{i}(a)$ which also do not form a bounding pair. 
Note that since $a$ and $b$ are nonseparating and do not form a bounding pair, the cut surfaces $S \setminus \{\gamma_1, \gamma_2\}$ and $S\setminus \{a, b\}$ are necessarily homeomorphic. 
So, it follows from \Cref{Lem:ChangeOfCoordiantes} that there exists an element $\psi_{f} \in \Map(S_g)$ such that $\psi_{f}(\gamma_1) = a$ and $\psi_{f}(\gamma_2) = b$.
Hence, it follows that $\psi_{f}^{-1} \phi^{i} \psi_{f}(\gamma_1) = \gamma_2$. 
As the pairs $\{x_1, x_3\}$ and $\{x_2, \alpha_2\}$ consist of disjoint curves which do not form a bounding pair, a similar argument shows that there exist elements $\psi_g$ and  $\psi_h$ in  $\Map(S_g)$ and $j,k \in \mathbb{Z}$ such that $\psi_{g}^{-1} \phi^{j} \psi_{g}(x_3) = x_1$ and $\psi_{h}^{-1} \phi^{k} \psi_{h}(x_2) = \alpha_2$. 

Therefore, the elements $f = \psi_f^{-1} \phi^i \phi_f$, $g = \psi_g^{-1} \phi^j \psi_g$, and $h = \psi_h^{-1} \phi^k \psi_h$ satisfy the hypothesis of \Cref{Lem:LanternToTwist}, and so $T_{\alpha_1}$ can be written in a product of 
six elements which are conjugate to $\phi$. Suppose now that $c$ is any other nonseparating simple closed curve in $S_g$. By the change of coordinates principle, there is a homeomorphism $\psi_c$ which maps $c$ to $\alpha_1$. Therefore, $T_c = T_{\phi_c^{-1}(\alpha_1)} = \psi_c^{-1} T_{\alpha_1} \psi_c$. Thus, the Dehn twist about any fixed nonseparating, simple closed curve in $S_g$ can be written as a product of 6 conjugates of $\phi$.

\end{proof}

As an immediate corollary, we get the following result about the stable specific torsion length of the Dehn twist about any nonseparating curve $c$. \\

\noindent \textbf{Corollary 1.5.}
\textit{Let $g \geq 3$, let $T_c$ be the Dehn twist about any fixed nonseparating curve $c$, and let $\phi \in \Map(S_g)$ be any torsion element which is not the hyperelliptic involution.
Then $\stl_\phi(T_c) \leq 6$.}\\

\bibliographystyle{plain}
\bibliography{bib}

\begin{thebibliography}{10}

\bibitem{AC}
Chloe~I. Avery and Lvzhou Chen.
\newblock Stable torsion length, 2021.
\newblock Preprint, \texttt{arXiv:2103.14116}.

\bibitem{BF}
Tara Brendle and Benson Farb.
\newblock Every mapping class group is generated by 6 involutions.
\newblock {\em Journal of Algebra}, 278(1):187--198, 2004.

\bibitem{Cal}
Danny Calegari.
\newblock Word length in surface groups with characteristic generating sets.
\newblock {\em Proceedings of the American Mathematical Society},
  136(7):2631–2637, 2008.

\bibitem{Cal2}
Danny Calegari.
\newblock Stable commutator length is rational in free groups.
\newblock {\em Journal of the American Mathematical Society}, 22(4):941–961,
  2009.

\bibitem{CF}
Danny Calegari and Koji Fujiwara.
\newblock Stable commutator length in word-hyperbolic groups.
\newblock {\em Journal of the American Mathematical Society}, 4(1):59–90,
  2010.

\bibitem{Chen}
Lvzhou Chen.
\newblock Scl in graphs of groups.
\newblock {\em Inventiones mathematicae}, 2021.

\bibitem{Dehn1}
M.~Dehn.
\newblock Lectures on group theory.
\newblock {\em Papers on Group Theory and Topology}, page 5–46, 1987.

\bibitem{Dugger}
Daniel Dugger.
\newblock Involutions on surfaces.
\newblock {\em Journal of Homotopy and Related Structures}, 14(4):919--992,
  2019.

\bibitem{FM}
B.~Farb and D.~Margalit.
\newblock {\em A primer on mapping class groups}.
\newblock Princeton University Press, 2012.

\bibitem{Humph}
Stephen~P. Humphries.
\newblock Generators for the mapping class group.
\newblock In {\em {Fenn R. (eds) Topology of Low-Dimensonal Manifolds. Lecture
  Notes in Mathematics}}, volume 722. Springer, Berlin, Heidelberg, Germany,
  1979.

\bibitem{Kass}
Martin Kassabov.
\newblock Generating mapping class groups by involutions, 2003.
\newblock Preprint, \texttt{arXiv:0311455}.

\bibitem{Klein}
Felix Klein.
\newblock Ueber realit\"atsverh\"altnisse bei der einem beliebigen geschlechte
  zugeh\"origen normalcurve der $\phi$.
\newblock {\em Mathematische Annalen}, 42(1):1--29, 1893.

\bibitem{Kork}
Mustafa Korkmaz.
\newblock On a question of {B}rendle and {F}arb, 2003.
\newblock Preprint, \texttt{arXiv:0307146}.

\bibitem{Kots}
D.~Kotschick.
\newblock Quasi-homomorphisms and stable lengths in mapping class groups.
\newblock {\em Proceedings of the American Mathematical Society},
  132(11):3167–3175, 2003.

\bibitem{Kulk}
Ravi~S. Kulkarni.
\newblock Riemann surfaces admitting large automorphism groups.
\newblock In {\em Extremal Riemann surfaces (San Francisco, CA, 1995)}, pages
  63--79. American Mathematical Society, Providence, RI, 1997.

\bibitem{Lanier}
Justin Lanier.
\newblock Universal bounds for torsion generating sets of mapping class groups.
\newblock {\em
  https://justinlanier944297149.files.wordpress.com/2020/10/bounded.pdf}, pages
  1--13, 2020.

\bibitem{LM}
Justin Lanier and Dan Margalit.
\newblock Normal generators for mapping class groups are abundant, 2018.
\newblock Preprint, \texttt{arXiv:1805.03666}.

\bibitem{Lick}
W.~B.~R. Lickorish.
\newblock A finite set of generators for the homeotopy group of a 2-manifold.
\newblock {\em Mathematical Proceedings of the Cambridge Philosophical
  Society}, 60:269--278, 1964.

\bibitem{Luo}
Feng Luo.
\newblock Torsion elements in the mapping class group of a surface, 2000.
\newblock Preprint, \texttt{arXiv:0004048}.

\bibitem{Mac}
Colin Maclachlan.
\newblock Modulus space is simply-connected.
\newblock {\em Proceedings of the American Mathematical Society}, 29(1):85--86,
  1971.

\bibitem{MP}
John McCarthy and Athanase Papadopoulos.
\newblock Involutions in surface mapping class groups.
\newblock {\em Proceedings of the American Mathematical Society}, 33:275–290,
  1987.

\bibitem{Mond}
Naoyuki Monden.
\newblock Generating the mapping class group by torsion elements of small
  order.
\newblock {\em Mathematical Proceedings of the Cambridge Philosophical
  Society}, 154(1):41--62, 2013.

\bibitem{Waj}
Bronislaw Wajnryb.
\newblock Mapping class group of a surface is generated by two elements.
\newblock {\em Topology}, 35(2):377--383, 1996.

\bibitem{Ye21}
Shengkui Ye.
\newblock Length functions on groups and rigidity, 2021.
\newblock Preprint, \texttt{arXiv:2101.08902}.

\end{thebibliography}

\end{document}